\documentclass[reqno]{amsart}
\usepackage{amsmath, amsthm, amssymb, amstext}

\usepackage{hyperref,xcolor}
\hypersetup{
 pdfborder={0 0 0},
 colorlinks,
}

\usepackage{enumitem}
\newtheorem{theorem}{Theorem}
\newtheorem{remark}[theorem]{Remark}

\newtheorem{proposition}[theorem]{Proposition}

\DeclareMathOperator*{\dom}{dom}              %
\DeclareMathOperator*{\divergenz}{div}              %
\DeclareMathOperator*{\ints}{int}         %
\DeclareMathOperator*{\essinf}{ess ~inf}         %
\DeclareMathOperator*{\ww}{w}         %
\DeclareMathOperator*{\Ss}{S}         %

\newcommand{\N}{\mathbb{N}}

\newcommand{\R}{\mathbb{R}}

\newcommand{\Lp}[1]{L^{#1}(\Omega)}

\newcommand{\Wpzero}[1]{W^{1,#1}_0(\Omega)}

\newcommand{\lan}{\langle}
\newcommand{\ran}{\rangle}
\newcommand{\eps}{\varepsilon}
\newcommand{\ph}{\varphi}

\newcommand{\into}{\int_{\Omega}}
\newcommand{\weak}{\overset{\ww}{\to}}

\newcommand{\Linf}{L^{\infty}(\Omega)}
\newcommand{\close}{\overline{\Omega}}
\newcommand{\interior}{\ints \left(C^1_0(\overline{\Omega})_+\right)}
\newcommand{\interiorK}{\ints K_+}
\newcommand{\cprime}{$'$}

\renewcommand{\l}{\left}
\renewcommand{\r}{\right}

\numberwithin{theorem}{section}
\numberwithin{equation}{section}

\title[$(p,q)$-equations with singular and concave convex nonlinearities]{$(p,q)$-Equations with singular and concave convex nonlinearities}

\author[N.\,S.\,Papageorgiou]{Nikolaos S.\,Papageorgiou}
\address[N.\,S.\,Papageorgiou]{National Technical University, Department of Mathematics, Zografou Campus, Athens 15780, Greece}
\email{npapg@math.ntua.gr}

\author[P.\,Winkert]{Patrick Winkert}
\address[P.\,Winkert]{Technische Universit\"{a}t Berlin, Institut f\"{u}r Mathematik, Stra\ss e des 17.\,Juni 136, 10623 Berlin, Germany}
\email{winkert@math.tu-berlin.de}

\subjclass[2010]{Primary: 35J20, Secondary: 35J75, 35J92}
\keywords{Singular and concave-convex terms, nonlinear regularity theory, nonlinear maximum principle, strong comparison theorems, minimal positive solution}

\begin{document}

\begin{abstract}
      We consider a nonlinear Dirichlet problem driven by the $(p,q)$-Laplacian with $1<q<p$. The reaction is parametric and exhibits the competing effects of a singular term and of concave and convex nonlinearities. We are looking for positive solutions and prove a bifurcation-type theorem describing in a precise way the set of positive solutions as the parameter varies. Moreover, we show the existence of a minimal positive solution and we study it as a function of the parameter.
\end{abstract}

\maketitle

\section{Introduction}

Let $\Omega \subseteq \R^N$  be a bounded domain with a $C^2$-boundary $\partial \Omega$. In this paper, we study the following parametric Dirichlet $(p,q)$-equation
\begin{align}\tag{P$_\lambda$}\label{problem}
  \begin{split}
    &-\Delta_p u-\Delta_q u = \lambda \left[u^{-\eta}+a(x) u^{\tau-1}\right]+f(x,u)\quad \text{in } \Omega \\
    &u\big|_{\partial \Omega}=0, \quad u>0, \quad \lambda>0,\quad 1<\tau<q<p, \quad 0<\eta<1.
   \end{split}
\end{align}
For $r\in (1,\infty)$ we denote by $\Delta_r$ the $r$-Laplace differential operator defined by
\begin{align*}
    \Delta_r u =\divergenz \left(|\nabla u|^{r-2} \nabla u\right)\quad\text{for all }u \in \Wpzero{r}.
\end{align*}
The perturbation in problem \eqref{problem}, namely $f\colon \Omega\times\R\to\R$, is a Carath\'{e}odory function, that is, $f$ is measurable in the first argument and continuous in the second one. We suppose that $f(x,\cdot)$ is $(p-1)$-superlinear near $+\infty$ but it does not satisfy the well-known Ambrosetti-Rabinowitz condition which we will write AR-condition for short. Hence, we have in problem \eqref{problem} the combined effects of singular terms (the function $s\to \lambda s^{-\eta}$), of sublinear (concave) terms (the function $s\to \lambda s^{\tau-1}$ since $1<\tau<q<p$) and of superlinear (convex) terms (the function $s\to f(x,s)$). For the precise conditions on $f$ we refer to hypotheses H($f$) in Section \ref{section_2}. Consider the following two functions (for the sake of simplicity we drop the $x$-dependence)
\begin{align*}
f_1(s)=\left(s^+\right)^{r-1}, \quad p<r<p^*, \qquad
f_2(s) =
\begin{cases}
\left(s^+\right)^l&\text{if }s\leq 1,\\
s^{p-1} \ln(s)+1&\text{if }1<s,
\end{cases}
\quad q<l.
\end{align*}
Both functions satisfy our hypotheses H($f$) but only $f_1$ satisfies the AR-condition.

We are looking for positive solutions and we establish the precise dependence of the set of positive solutions of \eqref{problem} on the parameter $\lambda>0$ as the latter varies. For the weight $a(\cdot)$ we suppose the following assumptions
\begin{enumerate}
	\item[H($a$):]
	$a\in\Linf$, $a(x)\geq a_0>0$ for a.\,a.\,$x\in\Omega$;
\end{enumerate}

The main result in this paper is the following one.
\begin{theorem}
	If hypotheses H($a$) and H($f$) hold, then there exists $\lambda^*\in (0,+\infty)$ such that
	\begin{enumerate}
		\item[(a)]
		for all $\lambda \in \l(0,\lambda^*\r)$, problem \eqref{problem} has at least two positive solutions
		\begin{align*}
		u_0, \hat{u} \in \interior \text{ with }u_0\leq \hat{u} \text{ and }u_0\neq \hat{u};
		\end{align*}
		\item[(b)]
		for $\lambda=\lambda^*$, problem \eqref{problem} has at least one positive solution $u^*\in \interior$;
		\item[(c)]
		for $\lambda>\lambda^*$, problem \eqref{problem} has no positive solution;
		\item[(d)]
		for every $\lambda \in \mathcal{L}=\l(0,\lambda^*\r]$, problem \eqref{problem} has a smallest positive solution $u^*_\lambda \in \interior$ and the map $\lambda \to u^*_\lambda$ from $\mathcal{L}$ into $C^1_0(\close)$ is strictly increasing, that is, $0<\mu<\lambda\leq \lambda^*$ implies $u^*_\lambda-u^*_\mu\in \interior$ and it is left continuous.
	\end{enumerate}
\end{theorem}

The study of elliptic problems with combined nonlinearities was initiated with the seminal paper of Ambrosetti-Brezis-Cerami \cite{Ambrosetti-Brezis-Cerami-1994} who studied semilinear Dirichlet equations driven by the Laplacian without any singular term. Their work has been extended to nonlinear problems driven by the $p$-Laplacian by Garc\'{\i}a Azorero-Peral Alonso-Manfredi \cite{Garcia-Azorero-Manfredi-Peral-Alonso-2000} and Guo-Zhang \cite{Guo-Zhang-2003}. In both works there is no singular term and the reaction has the special form
\begin{align*}
    x\to \lambda s^{\tau-1}+s^{r-1} \quad\text{for all $s \geq 0$ with }1<\tau<p<r<p^*,
\end{align*}
where $p^*$ is the critical Sobolev exponent to $p$ given by
\begin{align*}
    p^*=
    \begin{cases}
	\frac{Np}{N-p} &\text{if }p<N,\\
	+\infty &\text{if } N \leq p.
    \end{cases}
\end{align*}

More recently there have been generalizations involving more general nonlinear differential operators, more general concave and convex nonlinearities and different boundary conditions. We refer to the works of Papageorgiou-R\u{a}dulescu-Repov\v{s} \cite{Papageorgiou-Radulescu-Repovs-2017} for Robin problems and Papageorgiou-Winkert \cite{Papageorgiou-Winkert-2016}, Leonardi-Papa\-georgiou \cite{Leonardi-Papageorgiou-2019} and Marano-Marino-Papageorgiou \cite{Marano-Marino-Papageorgiou-2019} for Dirichlet problems. None of these works involves a singular term. Singular equations driven by the $p$-Laplacian and with a superlinear perturbation were investigated by Papageorgiou-Winkert \cite{Papageorgiou-Winkert-2019}.

We mention that $(p,q)$-equations arise in many mathematical models of physical processes. We refer to Benci-D'Avenia-Fortunato-Pisani \cite{Benci-DAvenia-Fortunato-Pisani-2000} for quantum physics and Cherfils-Il\cprime yasov \cite{Cherfils-Ilyasov-2005} for reaction diffusion systems.

Finally, we mention recent papers which are very close to our topic dealing with certain types of nonhomogeneous and/or singular problems. We refer to Papageorgiou-R\u{a}dulescu-Repov\v{s} \cite{Papageorgiou-Radulescu-Repovs-2020, Papageorgiou-Radulescu-Repovs-2019d}, Papageorgiou-Zhang \cite{Papageorgiou-Zhang-2021} and Ragusa-Tachikawa \cite{Ragusa-Tachikawa-2020}.

\section{Preliminaries and Hypotheses} \label{section_2}

We denote by $\Lp{p}$ $\l( \text{or } L^p\l(\Omega; \R^N\r)\r)$ and $\Wpzero{p}$ the usual Lebesgue and Sobolev spaces with their norms $\|\cdot\|_{p}$ and $\|\cdot\|$, respectively. By means of the Poincar\'{e} inequality we have
\begin{align*}
   \|u\|= \|\nabla u\|_p \quad \text{for all }u \in \Wpzero{p}.
\end{align*}
For $s \in \R$, we set $s^{\pm}=\max\{\pm s,0\}$ and for $u \in W^{1,p}_0(\Omega)$ we define $u^{\pm}(\cdot)=u(\cdot)^{\pm}$. It is known that
\begin{align*}
   u^{\pm} \in W^{1,p}_0(\Omega), \quad |u|=u^++u^-, \quad u=u^+-u^-.
\end{align*}
Furthermore, we need the ordered Banach space
\begin{align*}
   C^1_0(\overline{\Omega})= \left\{u \in C^1(\overline{\Omega}): u\big|_{\partial \Omega}=0 \right\}
\end{align*}
and its positive cone
\begin{align*}
   C^1_0(\overline{\Omega})_+=\left\{u \in C^1_0(\overline{\Omega}): u(x) \geq 0 \text{ for all } x \in \overline{\Omega}\right\}.
\end{align*}
This cone has a nonempty interior given by
\begin{align*}
   \ints \left(C^1_0(\overline{\Omega})_+\right)=\left\{u \in C^1_0(\overline{\Omega})_+: u(x)>0 \text{ for all } x \in \Omega \text{, } \frac{\partial u}{\partial n}(x)<0 \text{ for all } x \in \partial \Omega \right\},
\end{align*}
where $n(\cdot)$ stands for the outward unit normal on $\partial \Omega$. We will also use two more open cones. The first one is an open cone in the space $C^1(\close)$ and is defined by
\begin{align*}
   D_+&=\left\{u \in C^1(\overline{\Omega})_+: u(x)>0 \text{ for all } x\in \Omega, \ \frac{\partial u}{\partial n}\bigg|_{\partial \Omega \cap u^{-1}(0)}<0 \right\}.
\end{align*}
The second open cone is the interior of the order cone 
\begin{align*}
    K_+=\left\{u\in C_0(\close): u(x) \geq 0 \text{ for all }x\in\close\right\}
\end{align*}
of the Banach space
\begin{align*}
    C_0(\close)=\left\{u\in C(\close) : u\big|_{\partial\Omega}=0\right\}.
\end{align*}
We know that
\begin{align*}
    \interiorK=\left\{ u \in K_+: c_u \hat{d} \leq u \text{ for some }c_u>0\right\}
\end{align*}
with $\hat{d}(\cdot)=d(\cdot,\partial\Omega)$. Let $\hat{u}_1$ denote the positive $L^p$-normalized, that is, $\|\hat{u}_1\|_p=1$, eigenfunction of $\left(-\Delta_p,\Wpzero{p}\right)$. We know that $\hat{u}_1\in\interior$. From Papageorgiou-R\u{a}dulescu-Repov\v{s} \cite{Papageorgiou-Radulescu-Repovs-2019a} we have
\begin{align*}
    c_u \hat{d}\leq u \text{ for some }c_u>0 \quad\text{ if and only if }\quad \hat{c}_u\hat{u}_1 \leq u \text{ for some }\hat{c}_u>0.
\end{align*}

Given $u,v\in\Wpzero{p}$ with $u(x)\leq v(x)$ for a.\,a.\,$x\in\Omega$ we define
\begin{align*}
    [u,v]&=\left\{y\in\Wpzero{p}: u(x) \leq y(x) \leq v(x) \text{ for a.\,a.\,}x\in\Omega\right\},\\[1ex]
    \sideset{}{_{C^1_0(\close)}} \ints  [u,v]&=\text{the interior in } C^1_0(\close) \text{ of } [u,v]\cap C^1_0(\close),\\
    [u) & = \left\{y\in \Wpzero{p}: u(x) \leq y(x) \text{ for a.\,a.\,}x\in\Omega\right\}.
\end{align*}

If $h,g \in \Linf$, then we write $h \prec g$ if and only if for every compact set $K\subseteq \Omega$, there exists $c_K>0$ such that $c_K \leq g(x)-h(x)$ for a.\,a.\,$x\in K$. Note that if $h,g \in C(\Omega)$ and $h(x)<g(x)$ for all $x\in \Omega$, then $h\prec g$.

If $X$ is a Banach space and $\ph \in C^1(X)$, then we denote by $K_\ph$ the critical set of $\ph$, that is,
\begin{align*}
    K_\ph=\left\{ u\in X: \ph'(u)=0\right\}.
\end{align*}
Moreover, we say that $\ph$ satisfies the ``Cerami condition'', C-condition for short, if every sequence $\{u_n\}_{n \geq 1}\subseteq X$ such that $\{\ph(u_n)\}_{n\geq 1}\subseteq \R$ is bounded and 
\begin{align*}
    \left(1+\|u_n\|_X\right)\ph'(u_n) \to 0\quad\text{in }X^* \text{ as }n\to \infty,
\end{align*}    
admits a strongly convergent subsequence.

For every $r\in (1,\infty)$, let $A_r\colon \Wpzero{r}\to W^{-1,r'}(\Omega)=\Wpzero{r}^*$ with $\frac{1}{r}+\frac{1}{r'}=1$ be defined by
\begin{align*}
    \left\lan A_r(u), h\right\ran = \into |\nabla u|^{r-2} \nabla u \cdot \nabla h \,dx \quad\text{for all }u,h\in \Wpzero{r}.
\end{align*}
This operator has the following properties, see Gasi{\'n}ski-Papageorgiou \cite[p.\,279]{Gasinski-Papageorgiou-2016}. 

\begin{proposition}\label{proposition_1}
    The map $A_r\colon \Wpzero{r}\to W^{-1,r'}(\Omega)$ is bounded (that is, it maps bounded sets into bounded sets), continuous, strictly monotone (so maximal monotone) and of type ($\Ss$)$_+$, that is,
    \begin{align*}
	u_n\weak u \text{ in }\Wpzero{r} \quad \text{and}\quad \limsup_{n\to\infty} \lan A_r(u_n),u_n-u\ran \leq 0
    \end{align*}
    imply
    \begin{align*}
	u_n\to u \quad\text{in }\Wpzero{r}.
    \end{align*}
\end{proposition}

The hypotheses on the function $f(\cdot)$  are the following ones:

\begin{enumerate}
    \item[H($f$):]
	$f\colon \Omega\times\R\to\R$ is a Carath\'{e}odory function such that 
	\begin{enumerate}
	    \item[(i)]
		\begin{align*}
		  0\leq f(x,s) \leq c_1 \left[1+s^{r-1}\right] 
		\end{align*}
		for a.\,a.\,$x\in\Omega$, for all $s \geq 0$ with $c_1>0$ and $r\in (p,p^*)$;
	    \item[(ii)]
		if $F(x,s)=\int_0^sf(x,t)\,dt$, then
		\begin{align*}
		    \lim_{s\to+\infty} \frac{F(x,s)}{s^p}=+\infty\quad\text{uniformly for a.\,a.\,}x\in\Omega;
		\end{align*}
	    \item[(iii)]
		there exists $\mu \in \left((r-p)\max \left\{1,\frac{N}{p}\right\},p^*\right)$ with $\mu>\tau$ such that
		\begin{align*}
		    0<c_2 \leq \liminf_{s\to +\infty} \frac{f(x,s)s-pF(x,s)}{s^\mu} \quad\text{uniformly for a.\,a.\,}x\in\Omega;
		\end{align*}
	    \item[(iv)]
		\begin{align*}
		    \lim_{s\to 0^+} \frac{f(x,s)}{s^{q-1}}=0\quad\text{uniformly for a.\,a.\,}x\in\Omega;
		\end{align*}
	    \item[(v)]
		for every $\rho>0$ there exists $\hat{\xi}_\rho>0$ such that the function
		\begin{align*}
		    s \mapsto f(x,s)+\hat{\xi}_\rho s^{p-1}
		\end{align*}
		is nondecreasing on $[0,\rho]$ for a.\,a.\,$x\in\Omega$.
	\end{enumerate}
\end{enumerate}

\begin{remark}
    Since our aim is to produce positive solutions and all the hypotheses above concern the positive semiaxis $\R_+=[0,+\infty)$, we may assume, without any loss of generality, that
    \begin{align}\label{1}
	f(x,s)=0 \quad\text{for a.\,a.\,}x\in\Omega \text{ and for all }s \leq 0.
    \end{align}
\end{remark}

Note that  hypothesis H($f$)(iv) implies that $f(x,0)=0$ for a.\,a.\,$x\in\Omega$. From hypotheses H($f$)(ii), (iii) we infer that
\begin{align*}
    \lim_{s\to +\infty} \frac{f(x,s)}{s^{p-1}}=+\infty \quad\text{uniformly for a.\,a.\,}x\in\Omega.
\end{align*}
Therefore, the perturbation $f(x,\cdot)$ is $(p-1)$-superlinear for a.\,a.\,$x\in\Omega$. However, the superlinearity of $f(x,\cdot)$ is not expressed using the AR-condition which is common in the literature for superlinear problems. We recall that the AR-condition says that there exist $\beta>p$ and $M>0$ such that
\begin{align}
    0&<\beta F(x,s) \leq f(x,s)s \quad\text{for a.\,a.\,}x\in\Omega \text{ and for all }s\geq M,\label{2a}\\
    0&<\essinf_{x\in \Omega} F(x, M)\label{2b}.
\end{align}
In fact this is a uniliteral version of the AR-condition due to \eqref{1}. Integrating \eqref{2a} and using \eqref{2b} gives the weaker condition
\begin{align*}
    c_3 s^{\beta} \leq F(x,s) \quad\text{for a.\,a.\,}x\in\Omega, \text{ for all }x\geq M\text{ and for some }c_3>0,
\end{align*}
which implies
\begin{align*}
    c_3 s^{\beta-1} \leq f(x,s) \quad\text{for a.\,a.\,}x\in\Omega \text{ and for all }s \geq M.
\end{align*}
Hence, the AR-condition dictates that $f(x,\cdot)$ eventually has at least $(\beta-1)$-po\-ly\-no\-mi\-al growth. In the present work we replace the AR-condition by hypothesis H($f$)(iii) which includes in our framework also superlinear nonlinearities with slower growth near $+\infty$. 

Hypothesis H($f$)(v) is a one-sided H\"older condition. If $f(x,\cdot)$ is differentiable for a.\,a.\,$x\in\Omega$ and if for every $\rho>0$ there exists $c_\rho>0$ such that
\begin{align*}
    f'_s(x,s)s \geq -c_\rho s^{p-1}\quad\text{for a.\,a.\,}x\in\Omega \text{ and for all }0\leq s \leq \rho,
\end{align*}
then hypothesis H($f$)(v) is satisfied.

We introduce the following sets 
\begin{align*}
    \mathcal{L}&=\left\{\lambda>0: \text{problem \eqref{problem} admits a positive solution}\right\},\\
    \mathcal{S}_{\lambda}&=\left\{u: u \text{ is a positive solution of } \eqref{problem}\right\}.
\end{align*}

Moreover, we consider the following auxiliary Dirichlet problem
\begin{align}\tag{Q$_\lambda$}\label{problem2}
  \begin{split}
    &-\Delta_p u-\Delta_q u = \lambda a(x) u^{\tau-1}\quad \text{in } \Omega \\
    &u\big|_{\partial \Omega}=0, \quad u>0, \quad \lambda>0,\quad 1<\tau<q<p.
   \end{split}
\end{align}

\begin{proposition}\label{proposition_2}
    If hypothesis H($a$) holds, then for every $\lambda>0$ problem \eqref{problem2} admits a unique solution $\tilde{u}_\lambda \in \interior$.
\end{proposition}

\begin{proof}
    We consider the $C^1$-functional $\gamma_\lambda\colon\Wpzero{p}\to\R$ defined by
    \begin{align*}
	\gamma_\lambda(u)=\frac{1}{p}\|\nabla u\|_p^p +\frac{1}{q} \|\nabla u\|_q^q-\lambda \into a(x) \left(u^+\right)^{\tau}\,dx \quad\text{for all }u \in \Wpzero{p}.
    \end{align*}
    Since $\tau<q<p$ it is clear that $\gamma_\lambda\colon\Wpzero{p}\to\R$ is coercive and by the Sobolev embedding theorem, we see that $\gamma_\lambda\colon\Wpzero{p}\to\R$ is sequentially weakly lower semicontinuous. Hence, there exists $\tilde{u}_\lambda\in\Wpzero{p}$ such that
    \begin{align}\label{3}
	\gamma_\lambda\left(\tilde{u}_\lambda\right)=
	\min \left[\gamma_\lambda(u): u\in \Wpzero{p}\right].
    \end{align}
    If $u \in \interior$ and $t>0$ then
    \begin{align*}
	\gamma_\lambda (tu) = \frac{t^p}{p} \|\nabla u\|_p^p +\frac{t^q}{q} \|\nabla u\|^q_q-\frac{\lambda t^\tau}{\tau} \into a(x) u^2\,dx.
    \end{align*}
    Since $\tau <q<p$, choosing $t\in (0,1)$ small enough, we have $\gamma_\lambda (tu)<0$ and so,
    \begin{align*}
	\gamma_\lambda\left(\tilde{u}_\lambda\right )<0=\gamma_\lambda(0),
    \end{align*}
    see \eqref{3}, which shows that $\tilde{u}_\lambda\neq 0$. From \eqref{3} we know that $\gamma_\lambda'\left(\tilde{u}_\lambda\right)=0$, that is,
    \begin{align}\label{4}
	\lan A_p \left(\tilde{u}_\lambda\right),h\ran +\lan A_q\left(\tilde{u}_\lambda\right),h\ran=\lambda \into a(x) \left(\tilde{u}_\lambda^+\right)^{\tau -1}h\,dx\quad\text{for all }h\in\Wpzero{p}.
    \end{align}
    Choosing $h=-\tilde{u}^-_\lambda \in \Wpzero{p}$ in \eqref{4} gives
    \begin{align*}
	\left\|\nabla \tilde{u}^-_\lambda\right\|_p^p +\left\|\nabla \tilde{u}_\lambda^-\right\|_q^q =0,
    \end{align*}
    which shows that $\tilde{u}_\lambda \geq 0$ with $\tilde{u}_\lambda\neq 0$. Therefore, \eqref{4} becomes
    \begin{align*}
	&-\Delta_p \tilde{u}_\lambda -\Delta_q \tilde{u}_\lambda = \lambda a(x) \tilde{u}_\lambda^{\tau-1}\quad \text{in } \Omega, \qquad \tilde{u}_\lambda\big|_{\partial \Omega}=0.
    \end{align*}
    We know that $\tilde{u}_\lambda\in\Linf$, see, for example Marino-Winkert \cite{Marino-Winkert-2019}. Then, from the nonlinear regularity theory of Lieberman \cite{Lieberman-1991} we have that $\tilde{u}_\lambda \in C^1_0(\overline{\Omega})_+\setminus\{0\}$. Moreover, the nonlinear maximum principle of Pucci-Serrin \cite[pp.\,111, 120]{Pucci-Serrin-2007} implies that $\tilde{u}_\lambda \in \interior$.
    
    We still have to show that this positive solution is unique. Suppose that $\tilde{v}_\lambda \in \Wpzero{p}$ is another solution of \eqref{problem2}. As before we can show that $\tilde{v}_\lambda \in \interior$. We consider the integral functional $j\colon \Lp{1}\to \overline{\R}=\R\cup\{+\infty\}$ defined by
    \begin{align*}
	j(u)=
	\begin{cases}
	    \frac{1}{p}\left\|\nabla u^{\frac{1}{q}}\right\|_p^p +\frac{1}{q} \left\|\nabla u^{\frac{1}{q}}\right\|^q_q &\text{if } u\geq 0, \, u^{\frac{1}{q}} \in \Wpzero{p},\\
	    +\infty &\text{otherwise}.
	\end{cases}
    \end{align*}
    From D\'{\i}az-Sa\'{a} \cite[Lemma 1]{Diaz-Saa-1987} we see that $j$ is convex. Furthermore, applying Proposition 4.1.22 of Papageorgiou-R\u{a}dulescu-Repov\v{s} \cite[p.\,274]{Papageorgiou-Radulescu-Repovs-2019c}, we obtain that
    \begin{align*}
	\frac{\tilde{u}_\lambda}{\tilde{v}_\lambda}, \, \frac{\tilde{v}_\lambda}{\tilde{u}_\lambda}\in\Linf.
    \end{align*}
    We denote by
    \begin{align*}
	\dom j =\left\{u\in\Lp{1}: j(u)<+\infty\right\}
    \end{align*}
    the effective domain of $j$ and set $h=\tilde{u}_\lambda^q-\tilde{v}_\lambda^q$. One gets
    \begin{align*}
	\tilde{u}_\lambda^q-th \in \dom j \quad\text{and}\quad
	\tilde{v}_\lambda^q+th\in\dom j\quad\text{for all }t\in [0,1].
    \end{align*}
    Note that the functional  $j:\Lp{1}\to\overline{\R}$ is Gateaux differentiable at $\tilde{u}_\lambda^q$ and at $\tilde{v}_\lambda^q$ in the direction $h$. Using the nonlinear Green's identity, see Papageorgiou-R\u{a}dulescu-Repov\v{s} \cite[Corollary 1.5.16, p.\,34]{Papageorgiou-Radulescu-Repovs-2019c}, we obtain
    \begin{align*}
	j'\left(\tilde{u}_\lambda^q\right)(h)
	&=\frac{1}{q} \into \frac{-\Delta_p \tilde{u}_\lambda-\Delta_q \tilde{u}_\lambda}{\tilde{u}_\lambda^{q-1}}h\,dx =\frac{\lambda}{q}\into \frac{a(x)}{\tilde{u}_\lambda^{q-\tau}}h\,dx,\\
	j'\left(\tilde{v}_\lambda^q\right)(h)
	&=\frac{1}{q} \into \frac{-\Delta_p \tilde{v}_\lambda-\Delta_q \tilde{v}_\lambda}{\tilde{v}_\lambda^{q-1}}h\,dx =\frac{\lambda}{q}\into \frac{a(x)}{\tilde{v}_\lambda^{q-\tau}}h\,dx.
    \end{align*}
    The convexity of $j\colon\Lp{1}\to\overline{\R}$ implies the monotonicity of $j'$. Hence
    \begin{align*}
	0 \leq \frac{\lambda}{q}\into a(x) \left[\frac{1}{\tilde{u}_\lambda^{q-\tau}}-\frac{1}{\tilde{v}_\lambda^{q-\tau}} \right]\left[\tilde{u}_\lambda^q-\tilde{v}_\lambda^q\right]\,dx\leq 0,
    \end{align*}
    which implies $\tilde{u}_\lambda=\tilde{v}_\lambda$. Therefore, $\tilde{u}_\lambda\in\interior$ is the unique positive solution of the auxiliary problem \eqref{problem2}.
\end{proof}

This solution will provide a useful lower bound for the elements of the set of positive solutions $\mathcal{S}_\lambda$.

\section{Positive Solutions}

Let $\tilde{u}_\lambda\in\interior$ be the unique positive solution of \eqref{problem2}, see Proposition \ref{proposition_2}. Let $s>N$. Then $\tilde{u}_\lambda^s\in\ints K_+$ and so there exists $c_4>0$ such that
\begin{align*}
    \hat{u}_1 \leq c_4 \tilde{u}_\lambda^s,
\end{align*}
see Section \ref{section_2}. Hence
\begin{align*}
    \tilde{u}_\lambda^{-\eta} \leq c_5 \hat{u}_1^{-\frac{\eta}{s}}\quad\text{for some }c_5>0.
\end{align*}
Applying the Lemma of Lazer-McKenna \cite{Lazer-McKenna-1991} we have
\begin{align*}
    \hat{u}_1^{-\frac{\eta}{s}} \in \Lp{s}
\end{align*}
and thus
\begin{align}\label{5}
    \tilde{u}_\lambda^{-\eta} \in \Lp{s}.
\end{align}

We introduce the following modification of problem \eqref{problem} in which we have neutralized the singular term
\begin{align}\tag{P$_\lambda$'}\label{problem3}
  \begin{split}
    &-\Delta_p u-\Delta_q u = \lambda\tilde{u}_\lambda^{-\eta}+\lambda a(x)u^{\tau-1}+f(x,u)\quad \text{in } \Omega \\
    &u\big|_{\partial \Omega}=0, \quad u>0, \quad \lambda>0,\quad 1<\tau<q<p, \quad 0<\eta<1.
   \end{split}
\end{align}

Let $\psi_\lambda\colon\Wpzero{p}\to\R$ be the Euler energy functional of problem \eqref{problem3} defined by

\begin{align*}
    \psi_\lambda(u)
    &=\frac{1}{p}\|\nabla u\|_p^p +\frac{1}{q} \|\nabla u\|_q^q -\lambda \into \tilde{u}_\lambda^{-\eta} u\,dx\\ 
    &\quad -\frac{\lambda}{\tau} \into a(x)\left(u^+\right)^{\tau}\,dx-\into F(x,u^+)\,dx
\end{align*}
for all $u \in \Wpzero{p}$, see \eqref{5}. It is clear that $\psi_\lambda \in C^1(\Wpzero{p})$.

\begin{proposition}\label{proposition_3}
    If hypotheses H($a$) and H($f$) hold and if $\lambda>0$, then $\psi_\lambda$ satisfies the C-condition.
\end{proposition}

\begin{proof}
    Let $\{u_n\}_{n \geq 1} \subseteq \Wpzero{p}$ be a sequence such that
    \begin{align}
	&\left|\psi_\lambda (u_n)\right|\leq c_6 \quad\text{for all }n\in\N \text{ and for some }c_6>0,\label{6}\\
	&(1+\|u_n\|)\psi'_\lambda(u_n)\to 0 \quad\text{in }\Wpzero{p}^*=W^{-1,p'}(\Omega) \text{ with }\frac{1}{p}+\frac{1}{p'}=1.\label{7}
    \end{align}
    From \eqref{7} we have
    \begin{align}\label{8}
      \begin{split}
	&\left| \lan A_p(u_n),h\ran+\lan A_q(u_n),h\ran 
	-\lambda\into \tilde{u}_\lambda^{-\eta}h\,dx-\lambda\into a(x) \left(u_n^+\right)^{\tau-1}h\,dx\right.\\
	& \quad \left.-\into f\left(x,u_n^+\right)h\,dx\right| \leq \frac{\eps_n\|h\|}{1+\|u_n\|}\quad\text{for all } h\in\Wpzero{p} \text{ with }\eps_n\to 0^+.
      \end{split}
    \end{align}
    Choosing $h=-u_n^-\in\Wpzero{p}$ in \eqref{8} leads to
    \begin{align*}
	\left\|\nabla u_n^-\right\|_p^p \leq \eps_n\quad\text{for all }n \in \N,
    \end{align*}
    which implies
    \begin{align}\label{9}
	u_n^- \to 0\quad\text{in }\Wpzero{p} \text{ as }n\to\infty.
    \end{align}
    Combining \eqref{6} and \eqref{9} gives
    \begin{align}\label{10}
      \begin{split}
	& \left\|\nabla u_n^+\right\|_p^p+\frac{p}{q}\left\|\nabla u_n^+\right\|_q^q -\lambda p\into \tilde{u}_\lambda^{-\eta} u_n^+\,dx-\frac{\lambda p}{\tau} \into a(x) \left(u_n^+\right)^{\tau}\,dx\\ &-\into pF\left(x,u_n^+\right)\,dx
	\leq c_7 \quad \text{for all }n\in\N \text{ and for some }c_7>0.
      \end{split}
    \end{align}
    On the other hand, if we choose $h=u_n^+\in\Wpzero{p}$ in \eqref{8}, we obtain
    \begin{align}\label{11}
      \begin{split}
	& -\left\|\nabla u_n^+\right\|_p^p-\left\|\nabla u_n^+\right\|_q^q+\lambda \into \tilde{u}_\lambda^{-\eta} u_n^+\,dx+\lambda \into a(x) \left(u_n^+\right)^\tau\,dx\\
	& +\into f\left(x,u_n^+\right)u_n^+\,dx \leq \eps_n
	\quad\text{for all }n\in\N.
      \end{split} 
    \end{align}
    Adding \eqref{10} and \eqref{11} yields
    \begin{align}\label{12}
      \begin{split}
	&\into \left[ f\left(x,u_n^+\right)u_n^+-pF\left(x,u_n^+\right)\right]\,dx\\ 
	&\leq \lambda (p-1)\into \tilde{u}_\lambda^{-\eta}u_n^+\,dx+\lambda \left[\frac{p}{\tau}-1\right]\into a(x) \left(u_n^+\right)^\tau \,dx.
      \end{split}
    \end{align}
    By hypotheses H($f$)(i), (iii) we can find $c_8>0$ such that
    \begin{align*}
	\frac{c_2}{2}s^\mu-c_8 \leq f(x,s)s-pF(x,s)\quad\text{for a.\,a.\,}x\in\Omega \text{ and for all }s\geq 0.
    \end{align*}
    This implies
    \begin{align}\label{13}
	\frac{c_2}{2}s^\mu\left\|u_n^+\right\|_\mu^\mu -c_9 \leq \into \left[ f\left(x,u_n^+\right)u_n^+-pF\left(x,u_n^+\right)\right]\,dx
    \end{align}
    for some $c_9>0$ and for all $n\in\N$.
    
    Since $s>N$ we have $s'<N'\leq p^*$. Hence, $u_n^+\in\Lp{s'}$. Then, taking \eqref{5} along with H\"older's inequality into account, we get
    \begin{align}\label{14}
	\lambda [p-1]\into \tilde{u}_\lambda^{-\eta}u_n^+\,dx \leq c_{10} \left\|\tilde{u}_\lambda^{-\eta}\right\|_s \left\|u_n^+\right\|_{s'}
    \end{align}
    for some $c_{10}=c_{10}(\lambda)>0$ and for all $n\in\N$. Moreover, by hypothesis H($a$), we have
    \begin{align}\label{15}
	\lambda \left[\frac{p}{\tau}-1\right]\into a(x) \left(u_n^+\right)^\tau\,dx
	\leq c_{11} \left\|u_n^+\right\|_\tau^\tau 
    \end{align}
    for some $c_{11}=c_{11}(\lambda)>0$  and for all $n\in\N$.
    
    Now we choose $s>N$ large enough such that $s'<\mu$. Returning to \eqref{12}, using \eqref{13}, \eqref{14} as well as \eqref{15} and using the fact that $s', \tau <\mu$ by hypothesis H($f$)(iii) leads to
    \begin{align*}
	\left\|u_n^+\right\|_\mu^\mu \leq c_{12}\left[\left\|u_n^+\right\|_\mu +\left\|u_n^+\right\|_\mu^\tau+1\right]
    \end{align*}
    for some $c_{12}>0$ and for all $n\in\N$. Since $\tau<\mu$ we obtain
    \begin{align}\label{16}
	\left\{u_n^+\right\}_{n \geq 1} \subseteq \Lp{\mu} \text{ is bounded}.
    \end{align}

    Assume that $N\neq p$. From hypothesis H($f$)(iii) it is clear that we may assume $\mu<r<p^*$. Then there exists $t\in (0,1)$ such that
    \begin{align*}
	\frac{1}{r}=\frac{1-t}{\mu}+\frac{t}{p^*}.
    \end{align*}
    Taking the interpolation inequality into account, see Papageorgiou-Winkert \cite[Proposition 2.3.17, p.\,116]{Papageorgiou-Winkert-2018}, we have
    \begin{align*}
	\left\|u_n^+\right\|_r\leq \left\|u_n^+\right\|_\mu^{1-t} \left\|u_n^+\right\|^t_{p^*},
    \end{align*}
    which by \eqref{16} implies that
    \begin{align}\label{17}
	\left\|u_n^+\right\|_r^r\leq c_{13}\left\|u_n^+\right\|^{tr} 
    \end{align}
    for some $c_{13}>0$ and for all $n\in\N$.
    
    From hypothesis H($f$)(i) we know that
    \begin{align}\label{18}
	f(x,s)s \leq c_{14} \left[1+s^r\right] 
    \end{align}
    for a.\,a.\,$x\in\Omega$, for all $s\geq 0$ and for some $c_{14}>0$.
    We choose $h=u_n^+\in\Wpzero{p}$ in \eqref{8}, that is,
    \begin{align*}
	&\left\|\nabla u_n^+\right\|_p^p+\left\|\nabla u_n^+\right\|_q^q -\lambda\into \tilde{u}_\lambda^{-\eta}u_n^+\,dx-\lambda\into a(x) \left(u_n^+\right)^{\tau}\,dx\\
	& \quad -\into f\left(x,u_n^+\right)u_n^+\,dx\leq \eps_n\quad\text{for all } n\in\N.
    \end{align*}
    From this it follows by using \eqref{17}, \eqref{18} and $1<\tau<p<r$
    \begin{align}\label{19}
	\left\|u_n^+\right\|^p \leq c_{15}\left[ 1+\left\|u_n^+\right\|^{tr}\right]
    \end{align}
    for some $c_{15}>0$ and for all $n\in\N$. The condition on $\mu$, see hypothesis H($f$)(iii), implies that $tr<p$. Then from \eqref{19} we infer
    \begin{align}\label{20}
	\left\{u_n^+\right\}_{n\geq 1} \subseteq \Wpzero{p} \text{ is bounded.}
    \end{align}
    
    If $N=p$, then we have by definition $p^*=\infty$. The Sobolev embedding theorem  ensures that $\Wpzero{p}\hookrightarrow \Lp{\vartheta}$ for all $1\leq \vartheta<\infty$. So, in order to apply the previous arguments we need to replace $p^*$ by $\vartheta>r>\mu$ and choose $t \in (0,1)$ such that
    \begin{align*}
	\frac{1}{r}=\frac{1-t}{\mu}+\frac{t}{\vartheta},
    \end{align*}
    which implies
    \begin{align*}
	tr=\frac{\vartheta(r-\mu)}{\vartheta-\mu}.
    \end{align*}
    Note that $\frac{\vartheta(r-\mu)}{\vartheta-\mu}\to r-\mu<p$ as $\vartheta\to+\infty$. So, for $\vartheta>r$ large enough, we see that $tr<p$ and again \eqref{20} holds.
    
    From \eqref{9} and \eqref{20} we infer that
    \begin{align*}
	\left\{u_n\right\}_{n \geq 1} \subseteq \Wpzero{p} \text{ is bounded}.
    \end{align*}
    So, we may assume that
    \begin{align}\label{21}
	u_n\weak u \quad\text{in }\Wpzero{p} \quad\text{and}\quad
	u_n\to u \quad\text{in }\Lp{r}.
    \end{align}
    We choose $h=u_n-u\in\Wpzero{p}$ in \eqref{8}, pass to the limit as $n\to\infty$ and use the convergence properties in \eqref{21}. This gives
    \begin{align*}
	\lim_{n\to\infty} \left[ \lan A_p(u_n),u_n-u\ran+\lan A_q(u_n),u_n-u\ran\right]=0
    \end{align*}
    and since $A_q$ is monotone we obtain
    \begin{align*}
	\lim_{n\to\infty} \left[ \lan A_p(u_n),u_n-u\ran+\lan A_q(u),u_n-u\ran\right]\leq 0.
    \end{align*}
    By \eqref{20} we then conclude that
    \begin{align*}
	\lim_{n\to\infty} \lan A_p(u_n),u_n-u\ran\leq 0.
    \end{align*}
    Applying Proposition \ref{proposition_1} shows that $u_n\to u$ in $\Wpzero{p}$ and so we conclude that $\psi_\lambda$ satisfies the C-condition.
\end{proof}

\begin{proposition}\label{proposition_4}
    If hypotheses H($a$) and H($f$) hold, then there exists $\hat{\lambda}>0$ such that for every $\lambda \in \left(0,\hat{\lambda}\right)$ we can find $\rho_\lambda>0$ for which we have
    \begin{align*}
	\psi_\lambda(0)=0<\inf \left[\psi_\lambda(u): \|u\|=\rho_\lambda\right]=m_\lambda.
    \end{align*}
\end{proposition}

\begin{proof}
    Hypotheses H($f$)(i), (iv) imply that for a given $\eps>0$ we can find $c_{16}=c_{16}(\eps)>0$ such that
    \begin{align}\label{22}
	F(x,s)\leq \frac{\eps}{q}s^q +c_{16}s^r \quad\text{for a.\,a.\,}x\in\Omega \text{ and for all }s\geq 0.
    \end{align}
    Recall that $\tilde{u}_\lambda^{-\eta} \in \Lp{s}$ with $s>N$, see \eqref{5}. We choose $s>N$ large enough such that $s'<p^*$. Then, by H\"older's inequality, we have
    \begin{align}\label{23}
	\lambda \into \tilde{u}_\lambda^{-\eta}u\,dx \leq \lambda c_{17}\|u\|\quad\text{for some }c_{17}>0.
    \end{align}
    Moreover, one gets
    \begin{align}\label{24}
	\frac{\lambda}{\tau} \into a(x) |u|^\tau\,dx\leq \frac{\lambda \|a\|_\infty}{\tau} \|u\|^\tau.
    \end{align}
    Applying \eqref{22}, \eqref{23} and \eqref{24} leads to
    \begin{align}\label{24a}
     \psi_\lambda(u)
     \geq \frac{1}{p} \|\nabla u\|_p^p +\frac{1}{q} \left[\|\nabla u\|_q^q-\eps \|u\|_q^q\right]-c_{18} \left[ \|u\|^r+\lambda \left(\|u\|+\|u\|^\tau\right)\right]
    \end{align}
    for some $c_{18}>0$. Let $\hat{\lambda}_1(q)>0$ be the principal eigenvalue of $\left(-\Delta_q,\Wpzero{q}\right)$. Then, from the variational characterization of $\hat{\lambda}_1(q)$, see Gasi\'{n}ski-Papageorgiou \cite[p.\,732]{Gasinski-Papageorgiou-2006}, we obtain
    \begin{align*}
	\frac{1}{q} \left[\|\nabla u\|_q^q-\eps\|u\|_q^q\right] \geq \frac{1}{q} \left[1-\frac{\eps}{\hat{\lambda}_1(q)}\right] \|\nabla u\|_q^q.
    \end{align*}
    Choosing $\eps \in \left(0,\hat{\lambda}_1(q)\right)$ we infer that
    \begin{align}\label{25}
	\frac{1}{q} \left[\|\nabla u\|_q^q-\eps \|u\|_q^q\right]>0.
    \end{align}
    Since $1<\tau<r$, it holds
    \begin{align}\label{26}
	\|u\|^\tau \leq \|u\|+\|u\|^r.
    \end{align}
    Applying \eqref{25} and \eqref{26} to \eqref{24a} gives
    \begin{align}\label{27}
    \begin{split}
	\psi_\lambda(u)
	&\geq \frac{1}{p} \|u\|^p -c_{18} \left[ 2\lambda \|u\|+(\lambda+1)\|u\|^{r}\right]\\
    &\geq \l[ \frac{1}{p}  -c_{18} \left( 2\lambda \|u\|^{1-p}+(\lambda+1)\|u\|^{r-p}\r)\right] \|u\|^p.
    \end{split}
    \end{align}
    We consider now the function
    \begin{align*}
	k_\lambda(t)=2\lambda t^{1-p}+(\lambda+1)t^{r-p}\quad\text{for all }t>0.
    \end{align*}
    It is clear that $k_\lambda \in C^1(0,\infty)$ and since $1<p<r$ we see that
    \begin{align*}
	k_\lambda(t)\to +\infty \quad\text{as }t\to 0^+ \text{ and as }t\to +\infty.
    \end{align*}
    Hence, there exists $t_0>0$ such that
    \begin{align*}
	k_\lambda(t_0)=\min \left[k_\lambda(t): t>0\right],
    \end{align*}
    which implies that $k_\lambda'(t_0)=0$. Therefore,
    \begin{align*}
	2\lambda (p-1)t_0^{-p}=(r-p)(\lambda+1)t_0^{r-p-1}.
    \end{align*}
    From this we deduce that
    \begin{align*}
	t_0=t_0(\lambda)=\left[\frac{2\lambda(p-1)}{(r-p)(\lambda+1)}\right]^{\frac{1}{r-1}}.
    \end{align*}
    We have
    \begin{align*}
	k_\lambda(t_0)
	=2\lambda \frac{(r-p)(\lambda+1)^{\frac{p-1}{r-1}}}{(2\lambda(p-1))^{\frac{p-1}{r-1}}}
	+(\lambda+1) \frac{(2\lambda(p-1))^{\frac{r-p}{r-1}}}{((r-p)(\lambda+1))^{\frac{r-p}{r-1}}}.
    \end{align*}
    Since $1<p<r$ we see that
    \begin{align*}
	k_\lambda(t_0)\to 0 \quad\text{as }\lambda \to 0^+.
    \end{align*}
    Therefore, we can find $\hat{\lambda}>0$ such that
    \begin{align*}
	k_\lambda(t_0)<\frac{1}{pc_{18}}\quad\text{for all }\lambda \in \left(0,\hat{\lambda}\right).
    \end{align*}
    Then, by \eqref{27} we see that
    \begin{align*}
	\psi_\lambda(u)>0=\psi_\lambda(0)\quad\text{for all }\|u\|=t_0(\lambda)=\rho_\lambda \text{ and for all }\lambda \in \left(0,\hat{\lambda}\right).
    \end{align*}
\end{proof}

From hypothesis H($f$)(ii) we see that for every $u\in\interior$ we have
\begin{align}\label{28}
    \psi_\lambda(tu)\to -\infty \quad\text{as }t\to +\infty.
\end{align}

\begin{proposition}\label{proposition_5}
    If hypotheses H($a$) and H($f$) hold and if $\lambda \in \left(0,\hat{\lambda}\right)$, then problem \eqref{problem3} admits a solution $\overline{u}_\lambda \in \interior$.
\end{proposition}

\begin{proof}
    Propositions \ref{proposition_3}, \ref{proposition_4} and \eqref{28} permit the use of the mountain pass theorem. So, we can find $\overline{u}_\lambda \in \Wpzero{p}$ such that
    \begin{align}\label{29}
	\overline{u}_\lambda \in K_{\psi_\lambda}\quad\text{and}\quad \psi_\lambda(0)=0<m_\lambda\leq \psi_\lambda (\overline{u}_\lambda).
    \end{align}
    From \eqref{29} we see that $\overline{u}_\lambda\neq 0$ and $\psi_\lambda'(\overline{u}_\lambda)=0$, that is,
    \begin{align}\label{30}
      \begin{split}
	&\lan A_p(\overline{u}_\lambda),h\ran +\lan A_q(\overline{u}_\lambda),h\ran\\
	&=\lambda \into \tilde{u}_\lambda^{-\eta}h\,dx+\lambda \into a(x) \left(\overline{u}_\lambda^+\right)^{\tau-1}h\,dx
	+\into f\left(x,\overline{u}_\lambda^+\right)h\,dx
      \end{split}
    \end{align}
    for all $h\in \Wpzero{p}$. We choose $h=-\overline{u}_\lambda^-\in\Wpzero{p}$ in \eqref{30} which shows that
    \begin{align*}
	\left\|\overline{u}_\lambda^-\right\|^p \leq 0.
    \end{align*}
    Thus, $\overline{u}_\lambda\geq 0$ with $\overline{u}_\lambda\neq 0$.

    From \eqref{30} we know that $\overline{u}_\lambda$ is a positive solution of \eqref{problem3} with $\lambda \in\left(0,\hat{\lambda}\right)$. This means
    \begin{align*}
      &-\Delta_p \overline{u}_\lambda-\Delta_q \overline{u}_\lambda = \lambda\tilde{u}_\lambda^{-\eta}+\lambda a(x)\overline{u}_\lambda^{\tau-1}+f(x,\overline{u}_\lambda)\quad \text{in } \Omega, \quad \overline{u}_\lambda\big|_{\partial \Omega}=0.
    \end{align*}
    As before, see the proof of Proposition \ref{proposition_2}, using the nonlinear regularity theory, we have $\overline{u}_\lambda\in C^1_0(\overline{\Omega})_+\setminus\{0\}$. The nonlinear maximum principle, see Pucci-Serrin \cite[pp.\,111, 120]{Pucci-Serrin-2007} implies that $\overline{u}_\lambda \in \interior$.
\end{proof}

\begin{proposition}\label{proposition_6}
    If hypotheses H($a$) and H($f$) hold and if $\lambda \in \left(0,\hat{\lambda}\right)$, then $\tilde{u}_\lambda \leq \overline{u}_\lambda$.
\end{proposition}

\begin{proof}
    We introduce the Carath\'{e}odory function $g_\lambda\colon \Omega\times\R\to\R$ defined by
    \begin{align}\label{31}
	g_\lambda(x,s)=
	\begin{cases}
	    \lambda a(x) \left(s^+\right)^{\tau-1} &\text{if }s\leq \overline{u}_\lambda(x),\\
	    \lambda a(x) \overline{u}_\lambda(x)^{\tau-1} &\text{if }\overline{u}_\lambda(x)<s.
	\end{cases}
    \end{align}
    We set $G_\lambda(x,s)=\int^s_0g_\lambda(x,t)\,dt$ and consider the $C^1$-functional $\sigma_\lambda\colon\Wpzero{p}\to\R$ defined by
    \begin{align*}
	\sigma_\lambda (u)
	= \frac{1}{p}\|\nabla u\|_p^p+\frac{1}{q}\|\nabla u\|_q^q-\into G_\lambda(x,u)\,dx\quad\text{for all }u \in \Wpzero{p}.
    \end{align*}
    From \eqref{31} it is clear that $\sigma_\lambda\colon\Wpzero{p}\to\R$ is coercive. Moreover, by the Sobolev embedding, we have that $\sigma_\lambda\colon\Wpzero{p}\to\R$ is sequentially weakly lower semicontinuous. Then, by the Weierstra\ss-Tonelli theorem, we can find $\hat{u}_\lambda\in\Wpzero{p}$ such that
    \begin{align}\label{32}
	\sigma_\lambda\left(\hat{u}_\lambda\right)=
	\min \left[\sigma_\lambda(u):u\in\Wpzero{p}\right].
    \end{align}
    Since $\tau<q<p$, we have $\sigma_\lambda\left(\hat{u}_\lambda\right)<0=\sigma_\lambda(0)$ which implies $\hat{u}_\lambda\neq 0$.

    From \eqref{32} we have $\sigma_\lambda'\left(\hat{u}_\lambda\right)=0$, that is,
    \begin{align}\label{33}
	\lan A_p \left(\hat{u}_\lambda\right),h\ran +\lan A_q\left(\hat{u}_\lambda\right),h\ran 
	=\into g_\lambda \left(x,\hat{u}_\lambda\right)h\,dx\quad\text{for all }h\in\Wpzero{p}.
    \end{align}
    First, we choose $h=-\hat{u}_\lambda^-\in\Wpzero{p}$ in \eqref{33}. Then, by the definition of the truncation in \eqref{31} we easily see that $\|\hat{u}_\lambda^-\|^p \leq 0$ and so, $\hat{u}_\lambda\geq 0$ with $\hat{u}_\lambda\neq 0$. 
    
    Next, we choose $h=\left(\hat{u}_\lambda-\overline{u}_\lambda\right)^+\in\Wpzero{p}$ in \eqref{33} which gives, due to \eqref{31} and $f\geq 0$,
    \begin{align*}
	&\l\lan A_p\l(\hat{u}_\lambda\r),\left(\hat{u}_\lambda-\overline{u}_\lambda\right)^+ \r\ran
	+\l\lan A_q\l(\hat{u}_\lambda\r),\left(\hat{u}_\lambda-\overline{u}_\lambda\right)^+\r\ran\\
	&= \into \lambda a(x) \overline{u}_\lambda^{\tau-1} \left(\hat{u}_\lambda-\overline{u}_\lambda\right)^+\,dx\\
	& \leq \into \l[ \lambda \tilde{u}_\lambda^{-\eta}+\lambda a(x) \overline{u}_\lambda^{\tau-1}+f\l(x,\overline{u}_\lambda\r)\r] \left(\hat{u}_\lambda-\overline{u}_\lambda\right)^+\,dx\\
	&= \l\lan A_p\l(\overline{u}_\lambda\r),\left(\hat{u}_\lambda-\overline{u}_\lambda\right)^+\r\ran +\l\lan A_q\l(\overline{u}_\lambda\r),\left(\hat{u}_\lambda-\overline{u}_\lambda\right)^+\r\ran.
    \end{align*}
    This shows that $\hat{u}_\lambda\leq \overline{u}_\lambda$. We have proved that
    \begin{align*}
	\hat{u}_\lambda \in \left[0,\overline{u}_\lambda\right], \ \hat{u}_\lambda\neq 0.
    \end{align*}
    Hence, $\hat{u}_\lambda$ is a positive solution of \eqref{problem2} and due to Proposition \ref{proposition_2} we know that $\hat{u}_\lambda=\tilde{u}_\lambda \in \interior$. Therefore, $\tilde{u}_\lambda \leq \overline{u}_\lambda$ for all $\lambda\in \left(0,\hat{\lambda}\right)$.
\end{proof}

Now we are able to establish the nonemptiness of the set $\mathcal{L}$ (being the set of all admissible parameters) determine the regularity of the elements in the solution set $\mathcal{S}_\lambda$.

\begin{proposition}\label{proposition_7}
    If hypotheses H($a$) and H($f$) hold, then $\mathcal{L}\neq \emptyset$ and, for every $\lambda>0$, $\mathcal{S}_\lambda \subseteq \interior$.
\end{proposition}

\begin{proof}
    Let $\lambda \in \left(0,\hat{\lambda}\right)$. From Proposition \ref{proposition_6} we know that $\tilde{u}_\lambda\leq \overline{u}_\lambda$. So we can define the truncation $e_\lambda\colon\Omega\times\R\to\R$ of the reaction of problem \eqref{problem}
    \begin{align}\label{34}
      \begin{split}
	&e_\lambda(x,s)\\
	&=
	\begin{cases}
	    \lambda \left[ \tilde{u}_\lambda(x)^{-\eta}+a(x)\tilde{u}_\lambda(x)^{\tau-1}\right]+f\left(x,\tilde{u}_\lambda(x)\right) &\text{if }s<\tilde{u}_\lambda(x),\\
	    \lambda \left[ s^{-\eta}+a(x)s^{\tau-1}\right]+f(x,s) &\text{if }\tilde{u}_\lambda(x) \leq s\leq \overline{u}_\lambda(x),\\
	    \lambda \left[ \overline{u}_\lambda(x)^{-\eta}+a(x)\overline{u}_\lambda(x)^{\tau-1}\right]+f\left(x,\overline{u}_\lambda(x)\right) &\text{if }\overline{u}_\lambda(x) <s.
	\end{cases}
      \end{split}
    \end{align}
    This is a Carath\'{e}odory function. We set $E_\lambda(x,s)=\int^s_0 e_\lambda(x,t)\,dt$ and consider the $C^1$-functional $J_\lambda\colon\Wpzero{p}\to \R$ defined by
    \begin{align*}
	J_\lambda(u)=\frac{1}{p} \|\nabla u\|_p^p +\frac{1}{q} \|\nabla u\|_q^q-\into E_\lambda(x,u)\,dx \quad\text{for all }u \in \Wpzero{p}.
    \end{align*}
    From \eqref{34} we see that $J_\lambda\colon\Wpzero{p}\to \R$ is coercive and the Sobolev embedding theorem implies that $J$ is also sequentially weakly lower semicontinuous. Hence, its global minimizer $u_\lambda \in \Wpzero{p}$ exists, that is,
    \begin{align*}
	J_\lambda (u_\lambda)=\min\left[J_\lambda(u):u\in\Wpzero{p} \right].
    \end{align*}
    Hence, $J_\lambda'(u_\lambda)=0$ which means that
    \begin{align}\label{35}
	\left\lan A_p\l(u_\lambda\r),h\r\ran+\l\lan A_q\l(u_\lambda\r),h\r\ran =\into e_\lambda\l(x,u_\lambda\right)h\,dx \quad\text{for all } h \in \Wpzero{p}.
    \end{align}
    We choose $h=\left(u_\lambda-\overline{u}_\lambda\right)^+\in \Wpzero{p}$ in \eqref{35}. Then, by using \eqref{34} and Propositions \ref{proposition_6} and \ref{proposition_5} we obtain
    \begin{align*}
	& \l\lan A_p\l(u_\lambda\r),\left(u_\lambda-\overline{u}_\lambda\right)^+\r\ran
	+\l\lan A_q\l(u_\lambda\r),\left(u_\lambda-\overline{u}_\lambda\right)^+\r\ran\\
	& =\into \left(\lambda \l[\overline{u}_\lambda^{-\eta}+a(x)\overline{u}_\lambda^{\tau-1}\r]+f\l(x,\overline{u}_\lambda\r)\r) \left(u_\lambda-\overline{u}_\lambda\right)^+\,dx\\
	& \leq \into \l( \lambda \l[\tilde{u}_\lambda^{-\eta}+a(x)\overline{u}_\lambda^{\tau-1}\r]+f\l(x,\overline{u}_\lambda\r)\r) \left(u_\lambda-\overline{u}_\lambda\right)^+\,dx\\
	&=\l\lan A_p\l(\overline{u}_\lambda\r),\left(u_\lambda-\overline{u}_\lambda\right)^+\r\ran+\l\lan A_q\l(\overline{u}_\lambda\r),\left(u_\lambda-\overline{u}_\lambda\right)^+\r\ran.
    \end{align*}
    This shows that $u_\lambda \leq \overline{u}_\lambda$.
    
    Next, we choose $h=\left(\tilde{u}_\lambda-u_\lambda\right)^+\in\Wpzero{p}$ in \eqref{35}. Then, by \eqref{34} and hypotheses H($a$) as well as H($f$)(i) it follows
    \begin{align*}
	&\l\lan A_p \l(u_\lambda\r),\left(\tilde{u}_\lambda-u_\lambda\right)^+\r\ran+\l\lan A_q\l(u_\lambda\r),\left(\tilde{u}_\lambda-u_\lambda\right)^+\r\ran\\
	& =\into \l(\lambda \l[ \tilde{u}^{-\eta}+a(x)\tilde{u}_\lambda^{\tau-1}\r]+f\l(x,\tilde{u}_\lambda\r)\r) \left(\tilde{u}_\lambda-u_\lambda\right)^+\,dx\\
	&\geq \into \lambda \tilde{u}_\lambda^{-\eta} \left(\tilde{u}_\lambda-u_\lambda\right)^+\,dx\\
	&= \l \lan A_p\l(\tilde{u}_\lambda\r),\left(\tilde{u}_\lambda-u_\lambda\right)^+\r\ran+\l\lan A_q\l(\tilde{u}_\lambda\r),\left(\tilde{u}_\lambda-u_\lambda\right)^+\r\ran.
    \end{align*}
    Hence, $\tilde{u}_\lambda \leq u_\lambda$ and so we have proved that $u_\lambda \in \left[\tilde{u}_\lambda, \overline{u}_\lambda\r]$. Then, with view to \eqref{34} and \eqref{35}, we see that $u_\lambda$ is a positive solution of \eqref{problem} for $\lambda \in \l(0,\hat{\lambda}\r)$. In particular, we have
    \begin{align*}
	-\Delta_p u_\lambda(x)-\Delta_qu_\lambda(x)=\lambda u_\lambda(x)^{-\eta}+a_\lambda(x)u_\lambda(x)^{\tau-1}+f(x,u_\lambda(x))\quad\text{for a.\,a.\,}x\in\Omega.
    \end{align*}
    The nonlinear regularity theory, see Lieberman \cite{Lieberman-1991}, and the nonlinear maximum principle, see Pucci-Serrin \cite[pp.\,111 and 120]{Pucci-Serrin-2007}, imply that $u_\lambda \in \interior$.
    
    Concluding we can say that $\l(0,\hat{\lambda}\r)\subseteq \mathcal{L}$ which means that $\mathcal{L}$ is nonempty. Moreover, for all $\lambda>0$, $\mathcal{S}_\lambda\subseteq \interior$.
\end{proof}

Reasoning as in the proof of Proposition \ref{proposition_6} with $\overline{u}_\lambda$ replaced by $u \in \mathcal{S}_\lambda\subseteq \interior$, we obtain the following result.

\begin{proposition}
    If hypotheses H($a$) and H($f$) hold and if $\lambda \in \mathcal{L}$, then $\tilde{u}_\lambda\leq u$ for all $u \in \mathcal{S}_\lambda$. 
\end{proposition}

Moreover, the map $\lambda \to \tilde{u}_\lambda$ from $(0,+\infty)$ into $C^1_0(\close)$ exhibits a strong monotonicity property which we will use in the sequel.

\begin{proposition}\label{proposition_9}
    If hypotheses H($a$) holds and if $0<\lambda<\lambda'$, then $\tilde{u}_{\lambda'}-\tilde{u}_\lambda \in \interior$.
\end{proposition}

\begin{proof}
    Following the proof of Proposition \ref{proposition_6} we can show that
    \begin{align}\label{36}
	\tilde{u}_\lambda \leq \tilde{u}_{\lambda'}.
    \end{align}
    From \eqref{36} we have
    \begin{align}\label{37}	
      \begin{split}
	-\Delta_p \tilde{u}_\lambda-\Delta_q\tilde{u}_\lambda
	& =\lambda a(x) \tilde{u}_\lambda^{\tau-1}\\
	& =\lambda'a(x)\tilde{u}_\lambda^{\tau-1}-\l(\lambda'-\lambda\r)\tilde{u}_\lambda^{\tau-1}\\
	&\leq \lambda'a(x)\tilde{u}_{\lambda'}^{\tau-1}\\
	&= -\Delta_p \tilde{u}_{\lambda'}-\Delta_q\tilde{u}_{\lambda'}.
      \end{split}
    \end{align}
    Note that $0\prec\l(\lambda'-\lambda\r)\tilde{u}_\lambda^{\tau-1}$. So, from \eqref{37} and Gasi{\'n}ski-Papageorgiou \cite[Proposition 3.2]{Gasinski-Papageorgiou-2019} we have
    \begin{align*}
	\tilde{u}_{\lambda'}-\tilde{u}_\lambda \in \interior.
    \end{align*}
\end{proof}

Next we are going to show that $\mathcal{L}$ is an interval.

\begin{proposition}\label{proposition_10}
    If hypotheses H($a$) and H($f$) hold and if $\lambda \in\mathcal{L}$ and $\mu \in (0,\lambda)$, then $\mu \in \mathcal{L}$.
\end{proposition}

\begin{proof}
    Since $\lambda \in \mathcal{L}$ there exists $u_\lambda \in \mathcal{S}_\lambda \subseteq \interior$, see Proposition \ref{proposition_7}. From Propositions \ref{proposition_6} and \ref{proposition_9} we have
    \begin{align*}
	\tilde{u}_\mu \leq u_\lambda.
    \end{align*}
    We introduce the truncation function $\hat{k}_\mu\colon\Omega\times\R\to\R$ defined by
    \begin{align}\label{38}
		\begin{split}
			&\hat{k}_\mu(x,s)=\\
			&\begin{cases}
	    	\mu\left[\tilde{u}_\mu(x)^{-\eta}+a(x)u_\mu(x)^{\tau-1}\right] +f\l(x,u_\mu(x)\r)&\text{if }s<\tilde{u}_\mu(x),\\
		    \mu\left[s^{-\eta}+a(x)s^{\tau-1}\right] +f\l(x,s\r)&\text{if }\tilde{u}_\mu(x)\leq s \leq u_\lambda(x),\\
		    \mu\left[u_\lambda(x)^{-\eta}+a(x)u_\lambda(x)^{\tau-1}\right] +f\l(x,u_\lambda(x)\r)&\text{if }u_\lambda(x)<s,
			\end{cases}
		\end{split}
    \end{align}
    which is a Carath\'{e}odory function. We set $\hat{K}_\mu(x,s)=\int^s_0 \hat{k}_\mu(x,t)\,dt$ and consider the $C^1$-functional $\hat{\sigma}_\mu\colon\Wpzero{p}\to\R$ defined by
    \begin{align*}
	\hat{\sigma}_\mu(u)=\frac{1}{p}\|\nabla u\|^p_p+\frac{1}{q}\|\nabla u\|_q^q-\into \hat{K}_\mu(x,u)\,dx\quad\text{for all }u \in \Wpzero{p}.
    \end{align*}
    This functional is coercive because of \eqref{38} and sequentially weakly lower semicontinuous due to the Sobolev embedding theorem. Hence, there exists $u_\mu \in\Wpzero{p}$ such that
    \begin{align*}
	\hat{\sigma}_\mu(u_\mu)=\inf\left[\hat{\sigma}_\mu(u):\Wpzero{p}\right].
    \end{align*}
    Therefore, $\hat{\sigma}_\mu'(u_\mu)=0$ and so
    \begin{align}\label{39}
	\l\lan A_p\l(u_\mu\r),h\r\ran+\l\lan A_q\l(u_\mu\r),h\r\ran=\into \hat{k}_\mu\l(x,u_\mu\r)h\,dx
    \end{align}
    for all $h\in \Wpzero{p}$.
    We first choose $h=\l(u_\mu-u_\lambda\r)^+\in\Wpzero{p}$ in \eqref{39}. Then, by \eqref{38}, $\mu<\lambda$  and since $u_\lambda \in \mathcal{S}_\lambda$, we obtain
    \begin{align*}
	&\l\lan A_p\l(u_\mu\r),\l(u_\mu-u_\lambda\r)^+\r\ran+\l\lan A_q\l(u_\mu\r), \l(u_\mu-u_\lambda\r)^+\r\ran\\
	&=\into \left[\mu\l(u_\mu^{-\eta}+a(x)u_\lambda^{\tau-1}\r)+f\l(x,u_\lambda\r)\r] \l(u_\mu-u_\lambda\r)^+\,dx\\
	&\leq \into \l[\lambda \l(u_\lambda^{-\eta}+a(x)u_\lambda^{\tau-1}\r)+f\l(x,u_\lambda\r)\r] \l(u_\mu-u_\lambda\r)^+\,dx\\
	& =\l\lan A_p\l(u_\lambda\r), \l(u_\mu-u_\lambda\r)^+\r\ran+\l\lan A_q\l(u_\lambda\r),\l(u_\mu-u_\lambda\r)^+\r\ran.
    \end{align*}
    Hence, $u_\mu \leq v_\lambda$. In the same way, choosing $h=\l(\tilde{u}_\mu-u_\mu\r)^+ \in \Wpzero{p}$, we get from \eqref{38}, hypotheses H($a$), H($f$)(i) and Proposition \ref{proposition_2} that
    \begin{align*}
	&\l\lan A_p\l(u_\mu\r),\l(\tilde{u}_\mu-u_\mu\r)^+\r\ran+\l\lan A_q\l(u_\mu\r),\l(\tilde{u}_\mu-u_\mu\r)^+\r\ran\\
	& =\into \left[ \mu \l(\tilde{u}_\mu^{-\eta}+a(x)\tilde{u}_\mu^{\tau-1}\r)+f\l(x,\tilde{u}_\mu\r)\r] \l(\tilde{u}_\mu-u_\mu\r)^+ \,dx\\
	& \geq \into \mu \tilde{u}_\mu^{-\eta} \l(\tilde{u}_\mu-u_\mu\r)^+\,dx\\
	& = \l\lan A_p\l(\tilde{u}_\mu\r),\l(\tilde{u}_\mu-u_\mu\r)^+\r\ran +\l\lan A_q\l(\tilde{u}_\mu\r),\l(\tilde{u}_\mu-u_\mu\r)^+\r\ran.
    \end{align*}
    Thus, $\tilde{u}_\mu\leq u_\mu$. We have proved that
    \begin{align}\label{40}
	u_\mu\in\l[\tilde{u}_\mu,u_\lambda\r].
    \end{align}
    From \eqref{40}, \eqref{38} and \eqref{39} it follows that
    \begin{align*}
	u_\mu\in \mathcal{S}_\mu\subseteq \interior \text{ and so }\mu\in\mathcal{L}.
    \end{align*}
\end{proof}

Now we are going to prove that the solution multifunction $\lambda \to \mathcal{S}_\lambda$ has a kind of weak monotonicity property.

\begin{proposition}\label{proposition_11}
    If hypotheses H($a$) and H($f$) hold and if $\lambda\in\mathcal{L}, u_\lambda\in \mathcal{S}_\lambda \subseteq \interior$ and $\mu \in (0,\lambda)$, then $\mu \in \mathcal{L}$ and there exists $u_\mu\in \mathcal{S}_\mu\subseteq \interior$ such that
    \begin{align*}
	u_\lambda-u_\mu \in \interior.
    \end{align*}
\end{proposition}

\begin{proof}
    From Proposition \ref{proposition_10} and its proof we know that $\mu \in \mathcal{L}$ and that we can find $u_\mu \in \mathcal{S}_\mu\subseteq \interior$ such that $u_\mu \leq v_\lambda$. Let $\rho=\|u_\lambda\|_\infty$ and let $\hat{\xi}_\rho>0$ be as postulated by hypothesis H($f$)(v). Using $u_\mu\in\mathcal{S}_{\mu}$, hypotheses H($a$), H($f$)(v) and recalling that $\mu <\lambda$ we obtain
    \begin{align}\label{41}
      \begin{split}
	&-\Delta_p u_\mu-\Delta_q u_\mu +\hat{\xi}_\rho u_\mu^{p-1}-\mu u_\mu^{-\eta}\\
	& =\mu a(x) u_\mu^{\tau-1}+f(x,u_\mu)+\hat{\xi}_\rho u_\mu^{p-1}\\
	& =\lambda a(x) u_\mu^{\tau-1} +f(x,u_\mu)+\hat{\xi}_\rho u_\mu^{p-1}-(\lambda-\mu)a(x)u_\mu^{\tau-1}\\
	& \leq \lambda a(x) u_\lambda^{\tau-1} +f(x,u_\lambda)+\hat{\xi}_\rho u_\lambda^{p-1}\\
	& \leq -\Delta_p u_\lambda -\Delta_q u_\lambda +\hat{\xi}_\rho u_\lambda^{p-1}-\mu u_\lambda^{-\eta}.
      \end{split}
    \end{align}
    We have
    \begin{align*}
	0\prec(\lambda -\mu)a(x)u_\mu^{\tau-1}.
    \end{align*}
    Therefore, from \eqref{41} and Papageorgiou-Smyrlis \cite[Proposition 4]{Papageorgiou-Smyrlis-2015}, see also Proposition 7 in Papageorgiou-R\u{a}dulescu-Repov\v{s} \cite{Papageorgiou-Radulescu-Repovs-2019b}, we have
    \begin{align*}
	u_\lambda-u_\mu\in\interior.
    \end{align*}
\end{proof}

Let $\lambda^*=\sup \mathcal{L}$.

\begin{proposition}
    If hypotheses H($a$) and H($f$) hold, then $\lambda^*<\infty$.
\end{proposition}

\begin{proof}
    From hypotheses H($a$) and H($f$) we can find $\tilde{\lambda}>0$ such that
    \begin{align}\label{42}
	\tilde{\lambda} a(x) s^{\tau-1}+f(x,s) \geq s^{p-1} \quad\text{for a.\,a.\,}x\in\Omega \text{ and for all }s\geq 0.
    \end{align}
    Let $\lambda>\tilde{\lambda}$ and suppose that $\lambda\in\mathcal{L}$. Then we can find $u_\lambda\in \mathcal{S}_\lambda\subseteq \interior$. Consider a domain $\Omega_0\subset \subset \Omega$, that is, $\Omega_0\subseteq \Omega$ and $\overline{\Omega}_0\subseteq\Omega$, with a $C^2$-boundary $\partial\Omega_0$ and let $m_0=\min_{\overline{\Omega}_0}u_\lambda>0$. We set
    \begin{align*}
	m_0^\delta=m_0+\delta\quad\text{with}\quad \delta\in (0,1].
    \end{align*}
    Let $\rho=\max \l\{ \|u_\lambda\|_\infty, m_0^1\r\}$ and let $\hat{\xi}_\rho>0$ be as postulated by hypothesis H($f$)(v). Applying \eqref{42}, hypothesis H($f$)(v) and recalling that $u_\lambda\in \mathcal{S}_\lambda$ as well as $\tilde{\lambda}<\lambda$, we obtain
    \begin{align}\label{43}
      \begin{split}
	& -\Delta_p m_0^\delta-\Delta_q m_0^\delta +\hat{\xi}_\rho \left(m_0^\delta\r)^{p-1}-\tilde{\lambda} \left(m_0^\delta\r)^{-\eta}\\
	& \leq \hat{\xi}_\rho m_0^{p-1}+\chi(\delta) \quad\text{with }\chi(\delta)\to 0^+\text{ as }\delta\to 0^+\\
	&\leq \left[ \hat{\xi}_\rho+1\r] m_0^{p-1}+\chi(\delta)\\
	& \leq \tilde{\lambda} a(x) m_0^{\tau-1}+f(x,u_0)+\hat{\xi}_\rho m_0^{p-1}+\chi(\delta)\\
	&=\lambda a(x) m_0^{\tau-1}+f(x,m_0)+\hat{\xi}_\rho m_0^{p-1}-\l(\lambda-\tilde{\lambda}\r)m_0^{\tau-1} +\chi(\delta)\\
	& \leq \lambda a(x) m_0^{\tau-1} +f(x,m_0)+\hat{\xi}_\rho m_0^{p-1} \quad\text{for }\delta \in (0,1]\text{ small enough}\\
	& \leq \lambda a(x) u_\lambda^{\tau-1}+f(x,u_\lambda) +\hat{\xi}_\rho u_\lambda^{p-1}\\
	& =-\Delta_p u_\lambda -\Delta_q u_\lambda +\hat{\xi}_\rho u_\lambda^{p-1} -\lambda u_\lambda^{-\eta}\\
	& \leq -\Delta_p u_\lambda -\Delta_q u_\lambda +\hat{\xi}_\rho u_\lambda^{p-1}-\tilde{\lambda} u_\lambda^{-\eta}\quad \text{for a.\,a.\,}x\in \Omega_0.
      \end{split}
    \end{align}
    From \eqref{43} and Papageorgiou-R\u{a}dulescu-Repov\v{s} \cite[Proposition 6]{Papageorgiou-Radulescu-Repovs-2019b} we know that
    \begin{align*}
	u_\lambda-m_0^\delta\in D_+ \quad\text{for }\delta\in(0,1] \text{ small enough},
    \end{align*}
    a contradiction. Therefore, $\lambda^*\leq \tilde{\lambda}<\infty$.
\end{proof}

\begin{proposition}
    If hypotheses H($a$) and H($f$) hold and if $\lambda \in (0,\lambda^*)$, then problem \eqref{problem} has at least two positive solutions
    \begin{align*}
	u_0, \hat{u} \in \interior \text{ with }u_0 \leq \hat{u} \text{ and }u_0\neq \hat{u}.
    \end{align*}
\end{proposition}

\begin{proof}
    Let $\vartheta \in \l(\lambda,\lambda^*\r)$. According to Proposition \ref{proposition_11} we can find $u_\vartheta\in \mathcal{S}_\vartheta \subseteq \interior$ and $u_0\in \mathcal{S}_\lambda\subseteq\interior$ such that
    \begin{align*}
	u_\vartheta-u_0\in \interior.
    \end{align*}
    Recall that $\tilde{u}_\lambda \leq u_0$, see Proposition \ref{proposition_6}. Hence $u_0^{-\eta} \in \Lp{s}$ for all $s>N$, see \eqref{5}.
    
    We introduce the Carath\'{e}odory function $i_\lambda\colon\Omega\times\R\to\R$ defined by
    \begin{align}\label{44}
	i_\lambda(x,s)=
	\begin{cases}
	    \lambda \l[u_0(x)^{-\eta}+a(x)u_0(x)^{\tau-1}\r]+f(x,u_0(x))&\text{if }s\leq u_0(x),\\
	    \lambda\l[s^{-\eta}+a(x)s^{\tau-1}\r]+f(x,s)&\text{if }u_0(x)<s.
	\end{cases}
    \end{align}
    We set $I_\lambda(x,s)=\int^s_0 i_\lambda(x,t)\,dt$ and consider the $C^1$-functional $w_\lambda\colon\Wpzero{p}\to\R$ defined by
    \begin{align*}
	w_\lambda(u)=\frac{1}{p} \|\nabla u \|_p^p+\frac{1}{q}\|\nabla u\|_q^q-\into I_\lambda(x,u)\,dx\quad\text{for all }u\in\Wpzero{p}.
    \end{align*}
    Using \eqref{44} and the nonlinear regularity theory along with the nonlinear maximum principle we can easily check that
    \begin{align}\label{45}
	K_{w_\lambda}\subseteq [u_0)\cap \interior.
    \end{align}
    Then, from \eqref{44} and \eqref{45} it follows that, without any loss of generality, we may assume
    \begin{align}\label{46}
	K_{w_\lambda}\cap \l[u_0,u_\vartheta\r]=\{u_0\}.
    \end{align}
    Otherwise, on account of \eqref{44} and \eqref{45}, we see that we already have a second positive smooth solution of \eqref{problem} distinct and larger than $u_0$.
    
    We introduce the following truncation of $i_\lambda(x,\cdot)$, namely, $\hat{i}_\lambda\colon\Omega\times\R\to\R$ defined by
    \begin{align}\label{47}
	\hat{i}_\lambda(x,s)=
	\begin{cases}
	    i_\lambda(x,s) &\text{if }s\leq u_\vartheta(x),\\
	    i_\lambda(x,u_\vartheta(x)) &\text{if }u_\vartheta(x)<s,
	\end{cases}
    \end{align}
    which is a Carath\'{e}odory function. We set $\hat{I}_\lambda(x,s)=\int^s_0 \hat{i}_\lambda(x,t)\,dt$ and consider the $C^1$-functional $\hat{w}_\lambda\colon\Wpzero{p}\to\R$ defined by
    \begin{align*}
	\hat{w}_\lambda(u)=\frac{1}{p}\|\nabla u\|_p^p+\frac{1}{q}\|\nabla u\|_q^q-\into \hat{I}_\lambda (x,u)\,dx\quad\text{for all }u\in\Wpzero{p}.
    \end{align*}
    From \eqref{44} and \eqref{47} it is clear that $\hat{w}_\lambda$ is coercive and due to the Sobolev embedding theorem we know that $\hat{w}_\lambda$ is also sequentially weakly lower semicontinuous. Hence, we find $\hat{u}_0\in\Wpzero{p}$ such that
    \begin{align}\label{48}
	\hat{w}_\lambda\l(\hat{u}_0\r)=\min \l[\hat{w}_\lambda(u):u\in\Wpzero{p}\r].
    \end{align}
    It is easy to see, using \eqref{47}, that
    \begin{align}\label{49}
	K_{\hat{w}_\lambda}\subseteq \l[u_0,u_\vartheta\r]\cap \interior
    \end{align}
    and
    \begin{align}\label{50}
	\hat{w}_\lambda \big|_{\l[0,u_\vartheta\r]}
	=w_\lambda \big|_{\l[0,u_\vartheta\r]}, \quad
	\hat{w}'_\lambda \big|_{\l[0,u_\vartheta\r]}
	=w'_\lambda \big|_{\l[0,u_\vartheta\r]}.
    \end{align}
    From \eqref{48} we have $\hat{u}_0\in K_{\hat{w}'_\lambda}$ which by \eqref{46}, \eqref{49} and \eqref{50} implies that $\hat{u}_0=u_0$.
    
    Recall that $u_\vartheta-u_0\in\interior$. So, on account of \eqref{50}, we have that $u_0$ is a local $C^1_0(\close)$-minimizer of $w_\lambda$ and then $u_0$ is also a local $\Wpzero{p}$-minimizer of $w_\lambda$, see, for example Gasi{\'n}ski-Papageorgiou \cite{Gasinski-Papageorgiou-2012}.
    
    We may assume that $K_{w_\lambda}$ is finite, otherwise, we see from \eqref{45} that we already have an infinite number of positive smooth solutions of \eqref{problem} larger than $u_0$ and so we are done. From Papageorgiou-R\u{a}dulescu-Repov\v{s} \cite[Theorem 5.7.6, p.\,449]{Papageorgiou-Radulescu-Repovs-2019c} we find $\rho\in(0,1)$ small enough such that
    \begin{align}\label{51}
	w_\lambda(u_0)<\inf \l[w_\lambda(u):\|u-u_0\|=\rho\r]=m_\lambda.
    \end{align}
    If $u\in \interior$, then by hypothesis H($f$)(ii) we have
    \begin{align}\label{52}
	w_\lambda(tu)\to -\infty\quad\text{as }t\to+\infty.
    \end{align}
    Moreover, reasoning as in the proof of Proposition \ref{proposition_3}, we show that
    \begin{align}\label{53}
	w_\lambda \text{ satisfies the C-condition},
    \end{align}
    see also \eqref{44}. Then, \eqref{51}, \eqref{52} and \eqref{53} permit the use of the mountain pass theorem. So we can find $\hat{u}\in\Wpzero{p}$ such that
    \begin{align}\label{54}
	\hat{u}\in K_{w_\lambda}\subseteq [u_0)\cap\interior, \quad m_\lambda \leq w_\lambda\l(\hat{u}\r).
    \end{align}
    From \eqref{54}, \eqref{51} and \eqref{44} it follows that
    \begin{align*}
	\hat{u}\in \mathcal{S}_\lambda, \quad u_0\leq \hat{u}, \quad u_0\neq \hat{u}.
    \end{align*}
\end{proof}

\begin{remark}
    If $1<q=2\leq \lambda<p$, then, using the tangency principle of Pucci-Serrin \cite[p.\,35]{Pucci-Serrin-2007}, we can say that $\hat{u}-u_0\in\interior$.
\end{remark}

\begin{proposition}\label{proposition_14}
    If hypotheses H($a$) and H($f$) hold, then $\lambda^*\in\mathcal{L}$.
\end{proposition}

\begin{proof}
    Let $\lambda_n\nearrow\lambda^*$. With $\hat{u}_{n+1}\in \mathcal{S}_{\lambda_{n+1}}\subseteq \interior$ we introduce the following Carath\'{e}odory function (recall that $\tilde{u}_{\lambda_1}\leq \tilde{u}_{\lambda_n}\leq u$ for all $u\in \mathcal{S}_{\lambda_n}$ and for all $n\in \N$, see Propositions \ref{proposition_6} and \ref{proposition_9})
    \begin{align*}
    &\tilde{t}_n(x,s)=\\
	&\begin{cases}
	    \lambda_n\left [ \tilde{u}_{\lambda_1}(x)^{-\eta}+a(x)\tilde{u}_{\lambda_1}(x)^{\tau-1}\right]+f\l(x,\tilde{u}_{\lambda_1}(x)\r) &\text{if } s<\tilde{u}_{\lambda_1}(x)\\
	    \lambda_n\left [ s^{-\eta}+a(x)s^{\tau-1}\right]+f\l(x,s\r) &\text{if } \tilde{u}_{\lambda_1}(x)\leq s \leq \hat{u}_{n+1}(x)\\
	    \lambda_n\left [ \hat{u}_{n+1}(x)^{-\eta}+a(x)\hat{u}_{n+1}(x)^{\tau-1}\right]+f\l(x,\hat{u}_{n+1}(x)\r) &\text{if } \hat{u}_{n+1}(x)<s.
	\end{cases}
    \end{align*}
    Let $\tilde{T}_n(x,s)=\int^s_0 \tilde{t}_n(x,t)\,dt$ and consider the $C^1$-functional $\tilde{I}_n\colon\Wpzero{p}\to\R$ defined by
    \begin{align*}
	\tilde{I}_n(u)=
	\frac{1}{p}\|\nabla u\|_p^p+\frac{1}{q}\|\nabla u\|_q^q-\into \tilde{T}_n(x,u)\,dx\quad\text{for all }u \in\Wpzero{p}.
    \end{align*}
    Applying the direct method of the calculus of variations, see the definition of the truncation $\tilde{t}_n\colon\Omega\times\R\to\R$, we can find $u_n\in\Wpzero{p}$ such that
    \begin{align*}						
	\tilde{I}_n(u_n)=\min\l[\tilde{I}_n(u):u\in\Wpzero{p}\r].
    \end{align*}
    Hence, $\tilde{I}_n'(u_n)=0$ and so $u_n \in \l[\tilde{u}_{\lambda_1}, \hat{u}_{n+1}\r]\cap \interior$, see the definition of $\tilde{t}_n$. Moreover, $u_n\in \mathcal{S}_{\lambda_n}\subseteq \interior$.
    
    From Proposition \ref{proposition_2} we know that
    \begin{align*}
	\tilde{I}_n(u_n)\leq \tilde{I}_n\l(\tilde{u}_{\lambda_1}\r)<0.
    \end{align*}
    
    Now we introduce the truncation function $\hat{t}_n\colon\Omega\times\R\to\R$ defined by
    \begin{align}\label{56}
	\hat{t}_n(x,s)=
	\begin{cases}
	    \lambda_n\l[ \tilde{u}_{\lambda_1}(x)^{-\eta}+a(x)\tilde{u}_{\lambda_1}(x)^{\tau-1}\r]+f\l(x,\tilde{u}_{\lambda_1}(x)\r)&\text{if } s\leq \tilde{u}_{\lambda_1}(x),\\
	    \lambda_n\l[s^{-\eta}+a(x)s^{\tau-1}\r]+f(x,s)&\text{if } \tilde{u}_{\lambda_1}(x)<s.
	\end{cases}
    \end{align}
    We set $\hat{T}_n(x,s)=\int^s_0 \hat{t}_n(x,t)\,dt$ and consider the $C^1$-functional $\hat{I}_n\colon \Wpzero{p}\to\R$ defined by
    \begin{align*}
	\hat{I}_n(u)=\frac{1}{p}\|\nabla u\|_p^p+\frac{1}{q}\|\nabla u\|_q^q-\into \hat{T}_n(x,u)\,dx\quad\text{for all }u\in\Wpzero{p}.
    \end{align*}
    It is clear from the definition of the truncation $\tilde{t}_n\colon\Omega\times\R\to\R$ and \eqref{56} that
    \begin{align*}
	\hat{I}_n\big|_{\l[0,\hat{u}_{n+1}\r]}
	=\tilde{I}_n\big|_{\l[0,\hat{u}_{n+1}\r]}
	\quad\text{and}\quad
	\hat{I}'_n\big|_{\l[0,\hat{u}_{n+1}\r]}
	=\tilde{I}'_n\big|_{\l[0,\hat{u}_{n+1}\r]}.
    \end{align*}
    Then from the first part of the proof, we see that we can find a sequence $u_n\in \mathcal{S}_{\lambda_n}\subseteq \interior$, $n\in\N$, such that
    \begin{align}\label{57}
	\hat{I}_{n}(u_n)<0\quad\text{for all }n\in\N.
    \end{align}
    Moreover we have
    \begin{align}\label{58}
	\l\lan \hat{I}_n'(u_n),h\r\ran =0\quad\text{for all }h\in\Wpzero{p}\text{ and for all }n\in\N.
    \end{align}
    From \eqref{57} and \eqref{58}, reasoning as in the proof of Proposition \ref{proposition_3}, we show that
    \begin{align*}
	\l\{u_n\r\}_{n\geq 1} \subseteq \Wpzero{p} \text{ is bounded}.
    \end{align*}
    So we may assume that
    \begin{align*}
	u_n\weak u^* \text{ in }\Wpzero{p}\quad\text{and}\quad u_n\to u^*\text{ in }\Lp{r}.
    \end{align*}
    As before, see the proof of Proposition \ref{proposition_3}, using Proposition \ref{proposition_1} we show that
    \begin{align*}
	u_n\to u^* \text{ in }\Wpzero{p}.
    \end{align*}
    Then $u^* \in \mathcal{S}_{\lambda^*} \subseteq \interior$, recall that $\tilde{u}_{\lambda_1} \leq u_n$ for all $n\in\N$. This shows that $\lambda^*\in\mathcal{L}$.
\end{proof}

According to Proposition \ref{proposition_14} we have
\begin{align*}
    \mathcal{L}=(0,\lambda^*].
\end{align*}
The set $\mathcal{S}_\lambda$ is downward directed, see Papageorgiou-R\u{a}dulescu-Repov\v{s} \cite[Proposition 18]{Papageorgiou-Radulescu-Repovs-2019b}, that is, if $u, \hat{u}\in \mathcal{S}_\lambda$, we can find $\tilde{u}\in \mathcal{S}_\lambda$ such that $\tilde{u} \leq u$ and $\tilde{u}\leq \hat{u}$. Using this fact we can show that, for every $\lambda\in\mathcal{L}$, problem \eqref{problem} has a smallest positive solution.

\begin{proposition}\label{proposition_15}
    If hypotheses H($a$) and H($f$) hold and if $\lambda\in \mathcal{L}=(0,\lambda^*]$, then problem \eqref{problem} has a smallest positive solution $u_\lambda^*\in \interior$.
\end{proposition}

\begin{proof}
    Applying Lemma 3.10 of Hu-Papageorgiou \cite[p.\,178]{Hu-Papageorgiou-1997} we can find a decreasing sequence $\{u_n\}_{n \geq 1}\subseteq \mathcal{S}_\lambda$ such that
    \begin{align*}	
	\inf_{n\geq 1} u_n=\inf \mathcal{S}_\lambda.
    \end{align*}
    It is clear that $\{u_n\}_{n\geq 1}\subseteq \Wpzero{p}$ is bounded. Then, applying Proposition \ref{proposition_1}, we obtain
    \begin{align*}
	u_n\to u^*_\lambda \text{ in }\Wpzero{p}.
    \end{align*}
    Since $\tilde{u}_\lambda\leq u_n$ for all $n\in\N$ it holds $u^*_\lambda\in \mathcal{S}_\lambda$ and $u^*_\lambda=\inf \mathcal{S}_\lambda$.
\end{proof}

We examine the map $\lambda \to u^*_\lambda$ from $\mathcal{L}$ into $C^1_0(\close)$.

\begin{proposition}
    If hypotheses H($a$) and H($f$) hold, then the map $\lambda \to u^*_\lambda$ from $\mathcal{L}$ into $C^1_0(\close)$ is
    \begin{enumerate}
	\item[(a)]
	    strictly increasing, that is, $0<\mu<\lambda\leq \lambda^*$ implies $u^*_\lambda-u^*_\mu\in \interior$;
	\item[(b)]
	    left continuous.
    \end{enumerate}
\end{proposition}

\begin{proof}
    (a) Let $0<\mu<\lambda\leq \lambda^*$ and let $u^*_\lambda \in \interior$ be the minimal positive solution of problem \eqref{problem}, see Proposition \ref{proposition_15}. According to Proposition \ref{proposition_11} we can find $u_\mu \in \mathcal{S}_\mu\subseteq \interior$ such that $u^*_\lambda-u^*_\mu\in\interior$. Since $u^*_\mu \leq u_\mu$ we have $u^*_\lambda-u^*_\mu\in\interior$ and so, we have proved that $\lambda \to u^*_\lambda$ is strictly increasing.
    
    (b) Let $\{\lambda_n\}_{n \geq 1}\subseteq \mathcal{L}=(0,\lambda^*]$ be such that $\lambda_n \nearrow \lambda$ as $n\to \infty$. We have
    \begin{align*}
	\tilde{u}_{\lambda_1} \leq u^*_{\lambda_1} \leq u^*_{\lambda_n}\leq u^*_{\lambda^*} \quad\text{for all }n \in\N.
    \end{align*}
    Thus,
    \begin{align*}
	\l\{u^*_{\lambda_n}\r\}_{n\geq 1}\subseteq \Wpzero{p} \text{ is bounded}
    \end{align*}
    and so
    \begin{align*}
	\l\{u^*_{\lambda_n}\r\}_{n\geq 1}\subseteq \Linf \text{ is bounded},
    \end{align*}
    see Guedda-V\'{e}ron \cite[Proposition 1.3]{Guedda-Veron-1989}. 
    Therefore, we can find $\beta \in (0,1)$ and $c_{19}>0$ such that
    \begin{align*}	
	u^*_{\lambda_n} \in C^{1,\beta}_0(\close)\quad\text{and}\quad \l\|u^*_{\lambda_n}\r\|_{C^{1,\beta}_0(\close)}\leq c_{19} \quad\text{for all }n\in\N,
    \end{align*}
    see Lieberman \cite{Lieberman-1991}. The compact embedding of $C^{1,\beta}_0(\close)$ into $C^{1}_0(\close)$ and the monotonicity of $\l\{u^*_{\lambda_n}\r\}_{n\geq 1}$, see part (a), imply that
    \begin{align}\label{59}
	u^*_{\lambda_n}\to \hat{u}^*_{\lambda} \quad\text{in }C^{1}_0(\close).
    \end{align}
    If $\hat{u}^*_\lambda \neq u^*_\lambda$, then there exists $x_0\in\Omega$ such that
    \begin{align*}
	u^*_\lambda(x_0)<\hat{u}^*_\lambda(x_0) \quad\text{for all }n\in\N.
    \end{align*}
    From \eqref{59} we then conclude that
    \begin{align*}
	u^*_\lambda(x_0)<\hat{u}^*_{\lambda_n}(x_0)\quad\text{for all }n\in\N,
    \end{align*}
    which contradicts part (a). Therefore, $\hat{u}^*_\lambda = u^*_\lambda$ and so we have proved the left continuity of $\lambda \to u^*_\lambda$.
\end{proof}

%

\end{document}